\documentclass[12pt,reqno]{amsart}
\usepackage{graphicx}
\usepackage{amssymb}

\numberwithin{equation}{section}
\newtheorem{thm}{Theorem}[section]
\newtheorem{cor}[thm]{Corollary}
\newtheorem{lem}[thm]{Lemma}
\newtheorem{prop}[thm]{Proposition}

\theoremstyle{definition}
\newtheorem{defn}[thm]{Definition}
\theoremstyle{remark}
\newtheorem{rem}[thm]{Remark}

\numberwithin{equation}{section}

\newcommand\Hom{\operatorname{Hom}}

\newcommand\Ext{\operatorname{Ext}}

\newcommand\height{\operatorname{height}}

\newcommand\Rad{\operatorname{Rad}}
\newcommand\Ann{\operatorname{Ann}}

\newcommand\Supp{\operatorname{Supp}}

\newcommand\Tor{\operatorname{Tor}}

\newcommand{\qism}{\stackrel{\sim}{\longrightarrow}}
\newcommand\grade{\operatorname{grade}}
\newcommand\depth{\operatorname{depth}}

\newcommand\im{\operatorname{im}}

\begin{document}

\title[Endomorphism Ring]{On Endomorphism Rings of Local Cohomology Modules}%
\author[W. Mahmood ]{Waqas Mahmood \quad }
\address{Abdus Salam School of Mathematical Sciences, GCU, Lahore Pakistan}%
\email{ waqassms$@$gmail.com}

\thanks{This research was partially supported by Higher Education Commission, Pakistan}

\maketitle
\begin{abstract} Let $I\subset R$ be an ideal of a complete Cohen-Macaulay local ring of dimension $n$. We will show that the natural homomorphism $R\to \Hom_{R}(H^c_{I}(K_R),H^c_{I}(K_R))$ is an isomorphism provided that $H^i_{I}(R)= 0$ for all $i\neq c= \grade (I)$ where $K_R$ (resp. $H^i_{I}(\cdot)$) denote the canonical module (resp. $i$-th local cohomology with respect to the ideal $I$) of $R$. The same result is true for the Matlis dual of $H^c_{I}(K_R)$. Moreover for a finitely generated module $M$ over an arbitrary local ring $R$ and the two ideals $J\subset I$ of $R$ such that $\grade(I,M)= \grade(J,M)= c$ we will construct a natural homomorphism $\Hom_{R}(H^c_{J}(M),H^c_{J}(M))\to \Hom_{R}(H^c_{I}(M),H^c_{I}(M))$. Also we will show when this natural  homomorphism is an isomorphism. These results extends those of Hellus and  St\"uckrad (resp. Schenzel )(see \cite{h1} resp. \cite{p7}) shown in the case of $K_R= R$ (resp. $M= R$).
\end{abstract}
\section{Introduction}
Let $(R,\mathfrak m,k)$ be a Noetherian local ring. For an ideal $I\subset R$ we denote $H^i_I(R)$, $i\in \mathbb{Z}$, the local cohomology module of $R$ with respect to $I$ (see \cite{b} and \cite{goth} for the definition). For an $R$-module $M$ consider
the natural homomorphism $R\to \Hom_{R}(M,M),r\to f_r$ where $f_r(m)= rm$ for all $m\in M$ and $r\in R$. In general, this map is neither injective nor surjective. Recently there is a lot of work on the study of the endomorphism ring $\Hom_{R}(H^c_{I}(R), H^c_{I}(R))$ for $c= \grade (I)$. For instance Hellus and  St\"uckrad (see \cite{h1}) have shown that if $R$ is a complete local ring and $H^i_I(R)= 0$ for all $i\neq c$ then the natural homomorphism
\[
 R\to \Hom_{R}(H^c_{I}(R), H^c_{I}(R))
\]
is an isomorphism. They also have proved that the same result is true for the Matlis dual of the module $H^c_{I}(R).$ In the case of a local Gorenstein ring Schenzel (see \cite{p6}) has investigated that the endomorphism ring $\Hom_{R}(H^c_{I}(R), H^c_{I}(R))$ is a commutative ring. Moreover Eghbali and Schenzel (see \cite{e}) have shown that if $R$ is a complete local ring of dimension $n$ then the endomorphism ring $\Hom_{R}(H^n_{I}(R), H^n_{I}(R))$ is a commutative semi-local ring and finitely generated as an $R$-module. Here we will extend some of these results to the case of a finitely generated $R$-module $M$ (resp. the canonical module). In particular we prove the following result:

\begin{thm}\label{}
Let $R$ be a complete local ring of dimension $n$ and $I\subseteq R$ be an ideal. For a finitely generated $R$-module $M$ let $\grade(I,M)= c$. Then the following are equivalent:
\begin{itemize}
  \item [(a)] The natural homomorphism
\[
R\to \Hom_{R}(H^c_{I}(M),H^c_{I}(M))
\]
is an isomorphism.

 \item [(b)]  The natural homomorphism
\[
R\to \Hom_{R}(D(H^c_{I}(M)),D(H^c_{I}(M)))
\]
is an isomorphism where $D(\cdot):= \Hom_R(\cdot, E_R(k))$.
\end{itemize}

\end{thm}
This generalizes the result originally proved by Schenzel (see \cite{p1}) in case of $M=R$. In order to construct these natural homomorphisms we adopt the techniques developed in \cite{p1}. Moreover if $M=R$ a complete Cohen-Macaulay ring and $I\subseteq R$ is a cohomologically complete intersection ideal (i.e. $H^i_I(R)= 0$ for all $i\neq c= \grade (I)$) then we are able to prove the following result:
\begin{thm}\label{a}
Let $(R,\mathfrak{m})$ be a complete Cohen-Macaulay ring of dimension $n$. Suppose that $H^i_I(R)= 0$ for all $i\neq c= \grade (I)$ where $I\subseteq R$ is an ideal. Then the following hold:
\begin{itemize}
  \item [(a)] The natural homomorphism
\[
R\to \Hom_{R}(H^c_{I}(K_R),H^c_{I}(K_R))
\]
is an isomorphism and for all $i\neq c$
\[
\Ext_{R}^{i-c}(H^c_{I}(K_R),H^c_{I}(K_R))= 0.
\]

\item [(b)] The natural homomorphism
\[
R\to \Hom_{R}(D(H^c_{I}(K_R)),D(H^c_{I}(K_R)))
\]
is an isomorphism and for all $i\neq c$
\[
\Ext_{R}^{i-c}(D(H^c_{I}(K_R)),D(H^c_{I}(K_R)))= 0.
\]

\item [(c)] The natural homomorphism
\[
H^c_{I}(K_R)\otimes_R D(H^c_{I}(K_R))\to E
\]
is an isomorphism and for all $i\neq c$
\[
\Tor_{i-c}^R(H^c_I(K_R), D(H^c_{I}(K_R)))= 0.
\]
\end{itemize}
\end{thm}
In the above Theorem \ref{a} $K_R$ denotes the canonical module of $R$ (see \cite{her} for its definition and basic results). Moreover for a finitely generated $R$-module $M$ and the two ideals $J\subset I$ of $R$ such that $\grade(I,M)= \grade(J,M)= c$ there is a natural homomorphism
\[
 \Hom_{R}(H^c_{J}(M),H^c_{J}(M))\to \Hom_{R}(H^c_{I}(M),H^c_{I}(M)).
\]
In the following result there is a necessary condition such that this natural homomorphism becomes an isomorphism.
\begin{thm}\label{}
Let $J\subset I$ be two ideals of $(R,\mathfrak{m})$ such that $\grade(I,M)= \grade(J,M)= c$ for a finitely generated $R$-module $M$. Then the following are true:
\begin{itemize}
  \item [(a)] Suppose that $\Rad IR_{\mathfrak{p}}= \Rad JR_{\mathfrak{p}}$ for all prime ideals ${\mathfrak{p}}\in V(J)\cap \Supp_R(M)$ with $\depth_{R_\mathfrak{p}} (M_{\mathfrak{p}})\leq c$. Then the natural homomorphism \[
 \Hom_{R}(H^c_{J}(M),H^c_{J}(M))\to \Hom_{R}(H^c_{I}(M),H^c_{I}(M))
\]
is an isomorphism.
\item [(b)] Assume in addition that $R$ is complete Cohen-Macaulay and $H^i_{J}(R)= 0$ for all $i\neq c.$ Suppose that for $M= R$ the ideals $J\subset I$ satisfy the assumption of $(a)$. Then the natural homomorphism
\[
R\to \Hom_{R}(H^c_{I}(K_R),H^c_{I}(K_R))
\]
is an isomorphism.
\end{itemize}
\end{thm}
This generalizes the results of Schenzel (see \cite[Theorem 1.2]{p7}) to the case of a finitely generated $R$-modules (resp. the canonical module). Moreover corresponding results are true for $D(H^c_{I}(M))$ (resp. for $D(H^c_{I}(K_R))$) instead of $H^c_{I}(M)$ (resp. $K_R$).

\section{Preliminaries}
In this section we will fix the notation of the paper and summarize a few preliminaries and auxiliary results. Let $(R,\mathfrak m)$ be a commutative Noetherian local ring with $\mathfrak m$ the maximal ideal and $k= R/{\mathfrak m}$ the residue field. Furthermore we will denote the Matlis dual functor by $D(\cdot):= \Hom_R(\cdot, E)$ where $E= E_R(k)$ is the injective hull of $k$.

Here we will give a criterion of calculating grade of any ideal which is completely encoded in terms of vanishing of the Tor modules. Before this we need the following Proposition. Note that this was actually proved by Schenzel (see \cite[Thoerem 2.3]{p1}) for $M= R$. For the definition of grade and its properties we refer to \cite{her}.
\begin{prop}\label{31}
Let $I$ be an ideal of a ring $(R,\mathfrak m)$ and $M$ be a finitely generated $R$-module such that $c= \grade(I, M)$. Suppose that $N$ is any $R$-module with $\Supp_R(N)\subseteq V(I)$ then the following holds:
\begin{itemize}
  \item [(a)] There is an isomorphism
\[
\Ext^{c}_R(N, M)\cong \Hom_R(N, H^c_I(M))
\]
and for all $i< c$
\[
\Ext^{i}_R(N, M)=0.
\]
  \item [(b)] There is an isomorphism
  \[
\Tor_{c}^R(N, D(M))\cong N\otimes_R D(H^c_I(M))
\]
and for all $i< c$
\[
\Tor_{i}^R(N, D(M))= 0.
\]
\end{itemize}
\end{prop}
\begin{proof}
One can prove it easily by following the same steps as Schenzel used in his paper (see \cite[Thoerem 2.3]{p1}).
\end{proof}
\begin{lem}\label{4.5}
For a local ring $R$ let $M,N$ be any $R$-modules. Then for all $i\in \mathbb{Z}$ the following hold:
\begin{itemize}
\item[(1)] $\Ext^{i}_R(N,D(M))\cong D(\Tor_{i}^R(N, M))$.

\item[(2)] If $N$ is finitely generated then
\[
D(\Ext^{i}_R(N,M))\cong \Tor_{i}^R(N, D(M)).
\]
\end{itemize}
\end{lem}
\begin{proof}
For the proof see \cite[Example 3.6]{h}.
\end{proof}
In the next as an application of Proposition \ref{31} we will give a characterization of grade:
\begin{cor}
Fix the notation of Proposition \ref{31} then we have:
\[
\grade(I, M)= \inf\{i\in \mathbb{Z}: \Tor_{i}^R(R/I, D(M))\neq 0\}.
\]
\end{cor}
\begin{proof}
By Proposition \ref{31} it will be enough to show that $\Tor_{c}^R(R/I, D(M))\neq 0$. Since $\grade(I,M)= c$ by  Proposition \ref{31} $(b)$ we get that
\[
\Tor_{c}^R(R/I, D(M))\cong R/I\otimes_R D(H^c_I(M)).
\]
Then by Lemma \ref{4.5} this module is isomorphic to $D(\Hom_R(R/I, H^c_I(M)))$. Again by Proposition \ref{31} $(a)$ we conclude that
\[
\Tor_{c}^R(R/I, D(M))\cong D(\Ext^c_R(R/I, M))\neq 0
\]
(see \cite[Remark 3.11]{h}) which completes the proof. To this end recall that $\Ext^c_R(R/I, M)\neq 0$ for $c=\grade(I,M).$
\end{proof}

Note that the following result is an extension of the Local Duality Theorem which was proved by Grothendieck (see \cite{goth}). Here we will generalize it to cohomologically complete intersection ideals $I\subseteq R$.

\begin{lem}\label{b1} \text {(\bf Generalized Local Duality)}
Let $R$ be any local ring and $I$ be an ideal of $R$ such that $H^i_I(R)= 0$ for all $i\neq c= \grade(I).$ Then for any $R$-module $M$ and for all $i\in \mathbb{Z}$ we have
\begin{itemize}
  \item[(a)]
$\Tor_{c-i}^R(M, H^c_{I}(R)) \cong H^i_{I}(M)$.
\item[(b)]
$D(H^i_I(M)) \cong \Ext^{c-i}_R(M, D(H^c_{I}(R)))$.
\end{itemize}
\end{lem}

\begin{proof}
It is enough to prove the statement $(a)$ because of Lemma \ref{4.5}(1). To do this we will show that the family of the functors $\{\Tor_{c-i}^R(\cdot, H^c_{I}(R)): i\geq 0\}$ and $\{H^i_{I}(\cdot): i\geq 0\}$ are isomorphic. By \cite[Theorem 1.3.5]{b} we have to show the following conditions:
\begin{itemize}
  \item [(i)] For $i= c$ they are isomorphic.
  \item [(ii)] If $M$ is projective $R$-module then $\Tor_{c-i}^R(M, H^c_{I}(R))= 0= H^i_{I}(M)$ for all $i\neq c.$
  \item [(iii)] Both of the family of the functors $\{\Tor_{c-i}^R(\cdot, H^c_{I}(R)): i\geq 0\}$ and $\{H^i_{I}(\cdot): i\geq 0\}$ induces a connected long exact sequence of cohomologies from a given short exact sequence of $R$-modules $0\to N'\to N\to N''\to 0.$
\end{itemize}
Note that by definition $(iii)$ is satisfied. For $i= c$ note that our assumption $H^i_I(R)= 0$ for all $i\neq c= \grade(I)$ implies that the functor $H^c_I(\cdot)$ is right exact. So we have the following isomorphism
\[
H^c_{I}(R)\otimes_R M\cong H^c_{I}(M)
\]
which proves $(i)$. Now we only need to prove $(ii)$. Suppose firstly that $M= R$ then the result is true. Now suppose that $M$ is any projective $R$-module then $M\cong \oplus R.$ Hence $(ii)$ is also true. To this end recall that both the functors $\Tor_{c-i}^R(\cdot,H^c_{I}(R))$ and $H^i_I(\cdot)$ commutes with direct sums.
\end{proof}

In order to prove one of the main results we need the definition of the canonical module. To do this we recall the Local Duality Theorem (see \cite{goth}). Let $(R,\mathfrak m)$ denote a homomorphic image of a
local Gorenstein ring $(S,{\mathfrak n})$ with $\dim(S)=t.$ Let $N$ be a finitely generated $R$-module. Then by the Local Duality Theorem there is an isomorphism
\[
H^i_{\mathfrak m} (N)\cong D(\Ext^{t-i}_S(N, S))
\]
for all $i\in \mathbb N$ (see \cite{goth}). Now we are able to define the canonical module as follows:

\begin{defn}
Under the circumstances of the above Local Duality Theorem we define
\[
K_N:= \Ext^{t-r}_S(N, S), \dim(N)= r
\]
as the canonical module of $N$. It was introduced by Schenzel (see \cite{p4}) as
the generalization of the canonical module of a Cohen-Macaulay ring (see e.g. \cite{her}).
\end{defn}

\section{On The Endomorphism Ring}
For an $R$-module $M$ there is a natural homomorphism
\[
R\to \Hom_R(M,M), r\to f_r
\]
where $f_r: M\to M$, $m\to rm$, for all $m\in M$ and $r\in R.$ This homomorphism is in general neither injective nor surjective. In this section we will generalize one of the main results of
\cite[Theorem 2.2]{h1} and \cite[Theorem 1.2]{p7} to a finitely generated module (resp. the canonical module). To do this we need the definition of the truncation complex which is actually introduced by Schenzel (see \cite{p4}). Let $(R,\mathfrak m)$ be a local ring of dimension $n$ and $M$ be a finitely generated $R$-module. Let $I\subseteq R$
be an ideal of $R$ with $\grade(I, M) = c$. Suppose that $E^{\cdot}_R(M)$ is a minimal injective resolution of $M.$ Then it follows that
\[
E^{\cdot}_R(M)^i\cong\bigoplus\limits_{{\mathfrak p}\in \Supp_R(M)}E_R(R/{\mathfrak p})^{\mu_{i}({\mathfrak p}, M)}
\]
where $\mu_{i}({\mathfrak p}, M)= \dim_{k(\mathfrak{p})}(\Ext_{R_\mathfrak{p}}^i(k(\mathfrak{p}),M_\mathfrak{p}))$ and $k(\mathfrak{p})=R_\mathfrak{p}/\mathfrak{p}R_\mathfrak{p}$. Let $\Gamma_I(-)$ denote the section functor with support in $I$. Then $\Gamma_I(E^{\cdot}_R(M))^i= 0$ for all $i< c$ (see \cite[Definition 4.1]{p4} for more details).

\begin{defn} \label{2.2}
Let $C^{\cdot}_M(I)$ be the cokernel of the above embedding. It is called the truncation
complex of $M$ with respect to the ideal $I$. So there is a short exact sequence of complexes of $R$-modules
\[
0\rightarrow H^c_I(M)[-c]\rightarrow \Gamma_I(E^{\cdot}_R(M))\rightarrow C^{\cdot}_M(I)\rightarrow 0.
\]
\end{defn}
The following Lemma is the generalization of \cite[Theorem 1.1 and Corollary 1.2]{p1} to a finitely generated module.

\begin{lem}\label{3}
Let $R$ be a complete local ring of dimension $n$ and $I\subseteq R$ be an ideal. For a finitely generated $R$-module $M$ let $\grade(I,M)= c$. Then the following are true:
\begin{itemize}
  \item [(a)] There is an isomorphism
\[
\Hom_{R}(H^c_{I}(M),H^c_{I}(M))\cong \Hom_{R}(D(H^c_{I}(M)),D(H^c_{I}(M))).
\]

 \item [(b)] The natural homomorphism
\[
R\to \Hom_{R}(H^c_{I}(M),H^c_{I}(M))
\]
is an isomorphism if and only if the natural homomorphism
\[
R\to \Hom_{R}(D(H^c_{I}(M)),D(H^c_{I}(M)))
\]
is an isomorphism if and only if the natural homomorphism
\[
H^c_{I}(M)\otimes_R D(H^c_{I}(M))\to E
\]
is an isomorphism.
\end{itemize}
\end{lem}
\begin{proof}
Note that Lemma \ref{4.5} implies that there is an isomorphism
\[
\Hom_{R}(D(H^c_{I}(M)),D(H^c_{I}(M)))\cong D(D(H^c_{I}(M))\otimes_R H^c_{I}(M))
\]
and the module on the right side is isomorphic to
\[
D(\Tor_c^R(H^c_{I}(M),D(M)))\cong \Ext_R^c(H^c_{I}(M), M)
\]
(see Proposition \ref{31} $(b)$). To this end recall that $R$ is complete and by Matlis Duality the double Matlis dual of $M$ is isomorphic to $M$. Again by Proposition \ref{31} $(a)$ it induces that following isomorphism
\[
\Hom_{R}(D(H^c_{I}(M)),D(H^c_{I}(M)))\cong \Ext_R^c(H^c_{I}(M), M)\cong \Hom_{R}(H^c_{I}(M),H^c_{I}(M))
\]
which completes the proof of $(a)$. Note that by the isomorphism in $(a)$ the natural homomorphism
\[
R\to \Hom_{R}(H^c_{I}(M),H^c_{I}(M))
\]
is an isomorphism if and only if the natural homomorphism
\[
R\to \Hom_{R}(D(H^c_{I}(M)),D(H^c_{I}(M)))
\]
is an isomorphism. Moreover by Matlis duality a homomorphism $X\to Y$ between two $R$-modules is an isomorphism if and only if $D(X)\to D(Y)$ is an isomorphism. Then by Lemma \ref{4.5}(1) it follows that the natural homomorphism
\[
H^c_{I}(M)\otimes_R D(H^c_{I}(M))\to E
 \]
is an isomorphism if and only if $R\to \Hom_{R}(D(H^c_{I}(M),D(H^c_{I}(M))$ is an isomorphism. To this end recall that $R$ is complete and $\Hom_R(E,E)\cong R$. This completes the proof of the Lemma.
\end{proof}
\begin{lem}\label{3a}
Let $I$ be an ideal of a complete local ring $R$. Let $M$ be a finitely generated $R$-module such that $H^i_{I}(M)= 0$ for all $i\neq c=\grade(I,M)$. Then for each fixed integer $i\neq c$ the following  hold:
\begin{itemize}
\item [(a)] There are the following isomorphisms:
\begin{itemize}
\item [(i)] $\Ext_{R}^{i-c}(H^c_{I}(M),H^c_{I}(M))\cong \Ext_{R}^{i}(H^c_{I}(M),M).$
\item [(ii)] $ \Tor_{i-c}^R(H^c_I(M), D(H^c_{I}(M)))\cong \Tor_{i}^R(H^c_I(M), D(M)).$
\end{itemize}
\item [(b)] The following are equivalent:
\begin{itemize}
\item [(i)] $\Ext_{R}^{i-c}(H^c_{I}(M),H^c_{I}(M))=0.$
\item [(ii)] $\Ext_{R}^{i-c}(D(H^c_{I}(M)),D(H^c_{I}(M)))=0.$

\item [(iii)] $ \Tor_{i-c}^R(H^c_I(M), D(H^c_{I}(M)))=0.$
\end{itemize}
\end{itemize}
\end{lem}
\begin{proof}
Let $E^{\cdot}_R(M)$ be a minimal injective resolution of $M$. Since $H^i_{I}(M)= 0$ for all $i\neq c$ then the complex $\Gamma_I(E^{\cdot}_R(M))$ is a minimal injective resolution of $H_I^c(M)[-c]$. Moreover $\Supp_R(H_I^c(M))\subseteq V(I)$ so by \cite[Proposition 1.3]{pet1} there is an isomorphism
\[
\Ext_R^{i-c}(H^c_{I}(M),H^c_{I}(M))\cong \Ext_R^{i}(H^c_{I}(M),M)
\]
for each $i\in \mathbb{Z}$ which gives the first isomorphism of the statement $(a)$. The application of the Matlis dual functor $\Hom_R(\cdot, E)$ to the minimal injective resolution of $M$ (resp. $H^c_I(M))[-c]$) yields that the following complexes
\[
A:= D(E^{\cdot}_R(M)) \text { resp. } B:= D(\Gamma_I(E^{\cdot}_R(M)))
 \]
is a flat resolution of $D(M)$( resp. $D(H^c_I(M))[c]$). Since $\Supp_R(H_I^c(M))\subseteq V(I)$ so by \cite[Lemma 2.2]{p1} there is an isomorphism $H^c_I(M)\otimes_R A\cong  H^c_I(M)\otimes_R B$. This induces the following isomorphisms in cohomologies
\[
\Tor^R_{i-c}(H^c_I(M), D(H^c_{I}(M)))\cong \Tor^R_{i}(H^c_I(M), D(M)).
\]
for all $i\in \mathbb{Z}$. This completes the proof of $(a)$. By Lemma \ref{4.5} there is an isomorphism
\begin{equation}\label{wq}
\Ext_{R}^{i-c}(D(H^c_{I}(M)),D(H^c_{I}(M)))\cong D(\Tor^R_{i-c}(D(H^c_{I}(M)),H^c_I(M)))
\end{equation}
for all $i\in \mathbb{Z}$. Again by Lemma \ref{4.5} and Matlis Duality we get that
\begin{equation}\label{wq1}
D(\Tor_i^R(H^c_{I}(M),D(M)))\cong \Ext_R^i(H^c_{I}(M), M).
\end{equation}
Hence the equivalence of the vanishing of the statement $(b)$ is easy by virtue of $(a)$ and the isomorphisms \ref{wq}, \ref{wq1}.
\end{proof}

\begin{thm}\label{5}
Let $R$ be a complete Cohen-Macaulay ring of dimension $n$. Suppose that $H^i_I(R)= 0$ for all $i\neq c= \grade (I)$ where $I\subseteq R$ is an ideal. Then the following hold:
\begin{itemize}
  \item [(a)] The natural homomorphism
\[
R\to \Hom_{R}(H^c_{I}(K_R),H^c_{I}(K_R))
\]
is an isomorphism and for all $i\neq c$
\[
\Ext_{R}^{i-c}(H^c_{I}(K_R),H^c_{I}(K_R))= 0.
\]

\item [(b)] The natural homomorphism
\[
R\to \Hom_{R}(D(H^c_{I}(K_R)),D(H^c_{I}(K_R)))
\]
is an isomorphism and for all $i\neq c$
\[
\Ext_{R}^{i-c}(D(H^c_{I}(K_R)),D(H^c_{I}(K_R)))= 0.
\]

\item [(c)] The natural homomorphism
\[
H^c_{I}(K_R)\otimes_R D(H^c_{I}(K_R))\to E
\]
is an isomorphism and for all $i\neq c$
\[
\Tor_{i-c}^R(H^c_I(K_R), D(H^c_{I}(K_R)))= 0.
\]
\end{itemize}
\end{thm}
\begin{proof}
Firstly note that by Lemmas \ref{3} and \ref{3a} it is enough to prove the statement $(a)$. To this end note that the canonical module $K_R$ of $R$ exists and
\[
\grade (I, K_R)=\dim_R(K_R)-\dim_R(K_R/IK_R)
\]
(see \cite{her}). It follows that $c= \grade (I, K_R)$. Let $E^{\cdot}_R(K_R)$ be a minimal injective resolution of $K_R$ then apply the functor $\Hom_R( ., E^{\cdot}_R(K_R))$ to the short exact sequence of the truncation complex. It induces a short exact sequence of complexes of $R$-modules
\begin{gather*}
0\rightarrow \Hom_R(C^{\cdot}_{K_R}(I), E^{\cdot}_R(K_R))\to \Hom_R(\Gamma_I(E^{\cdot}_R(K_R)),
E^{\cdot}_R(K_R))\\
\to \Hom_R(H^c_I(K_R), E^{\cdot}_R(K_R))[c]\to 0.
\end{gather*}
Note that $ E^{\cdot}_R(K_R)$ is a complex of injective $R$-modules. By the definition of the section functor $\Gamma_I(\cdot)$ and a standard isomorphism
for the direct and inverse limits  the complex in the middle is isomorphic to
\[
\lim_{\longleftarrow}
\Hom_R(\Hom_R(R/I^r,E^{\cdot}_R(K_R)), E^{\cdot}_R(K_R))
\cong \lim_{\longleftarrow} (R/I^r\otimes_{R} \Hom_R(E^{\cdot}_R(K_R),
E^{\cdot}_R(K_R))).
\]
For the last isomorphism note that $R/I^r$ is a finitely generated $R$-module for all $r \geq 1.$
Let $X := \Hom_R(E^{\cdot}_R(K_R), E^{\cdot}_R(K_R))$ then there is a quasi isomorphism
\[
X \qism \Hom_R(K_R, E^{\cdot}_R(K_R)).
\]
By definition of Hom of complexes we have
\[
X^k = \prod_{i \in \mathbb{Z}}
\Hom_R(E^i_R(K_R), E^{i+k}_R(K_R)).
\]
Therefore $X^k, k \in \mathbb{Z},$ is a flat $R$-module. Since $R$ is Cohen-Macaulay so $H_\mathfrak{m}^i(R)= 0$ for all $i\neq n$. By the Generalized Local Duality (see Lemma \ref{b1}) this implies that $H^k(X) \cong \Ext^k_R(K_R,K_R)= 0$ for all $k \neq 0$ and $H^0(X) \cong \Hom_R(K_R,K_R)\cong R$ . Also note that
\[
E^{i}_R(K_R) \cong
\oplus_{\height{p} =i}E_R(R/\mathfrak{p}).
\]
To this end recall that $R$ is Cohen-Macaulay and $K_R$ is of finite injective dimension. Then it implies that $X^k= 0$ for all $k> 0$. Hence the complex $X$ becomes a flat resolution of $R$. So the cohomologies of the complex $\lim\limits_{\longleftarrow} (R/I^r\otimes_{R} X)$ are zero for all $i\neq 0$ and it is $R$ for $i= 0$ (since $R$ is complete). Moreover note that $H^i_I(K_R)= 0$ for all $i\neq c$. Then the complex $\Hom_R(C^{\cdot}_{K_R}(I), E^{\cdot}_R(K_R)) $ is an exact complex. Therefore from the long exact sequence of cohomology of the above short exact sequence it follows that the natural homomorphism
\[
R\to \Ext_{R}^{c}(H^c_{I}(K_R),K_R)
\]
is an isomorphism and for all $i\neq c$
\[
\Ext_{R}^{i}(H^c_{I}(K_R),K_R)= 0.
\]
This proves the statement $(a)$ by virtue of Proposition \ref{31} $(a)$ and Lemma \ref{3a} $(a)$.

\end{proof}
Note that the following Proposition is actually proved by Schenzel (see \cite[Theorem 1.2]{p7}) in case of a complete Gorenstein local ring. Here we will give a generalization of it to a finitely generated module $M$ over an arbitrary local ring $R$.

\begin{prop}\label{21.3}
Let $R$ be a local ring of dimension $n$ and $J\subset I$ be two ideals of $R$. Suppose that $\grade(I,M)=c= \grade(J,M)$ for a finitely generated $R$-module $M$. Then we have:
\begin{itemize}
  \item [(a)] There is a natural morphism
\[
\Hom_{R}(H^c_{J}(M),H^c_{J}(M))\to \Hom_{R}(H^c_{I}(M),H^c_{I}(M)).
\]
\item [(b)] Suppose that $\Rad IR_{\mathfrak{p}}= \Rad JR_{\mathfrak{p}}$ for all ${\mathfrak{p}}\in V(J)\cap \Supp_R(M)$ such that $\depth_{R_\mathfrak{p}} (M_{\mathfrak{p}})\leq c$ then the above natural homomorphism is an isomorphism.
\item [(c)] Suppose in addition $R$ is complete then there is a natural homomorphism
\[
\Hom_{R}(D(H^c_{J}(M)),D(H^c_{J}(M)))\to \Hom_{R}(D(H^c_{I}(M)),D(H^c_{I}(M)))
\]
and under the assumption of $(b)$ it is an isomorphism.
\end{itemize}
\end{prop}
\begin{proof}

Note that for each integer $s\geq 1$ there is a short exact sequence
\[
0\to I^s/J^s\to R/J^s\to R/I^s\to 0
\]
this is true because of $J\subset I$. Then the application of the functor $\Hom_{R}(\cdot, M)$ to this sequence induces the following exact sequence
\[
0\to \Ext^c_R(R/I^s, M)\to \Ext^c_R(R/J^s, M)\to \Ext^c_R(I^s/J^s,M)
\]
for all $s\geq 1$. Because of $\Supp_R(I^s/J^s) \subseteq V(J)$ it follows that $\Ext^{c-1}_R(I^s/J^s,M) = 0$ since $\grade (J,M) = c$ (see Proposition \ref{31}(a)).  Now pass to the direct limit of the last sequence we get that
\begin{equation}\label{ee}
0\to H^c_{I}(M)\to H^c_{J}(M)\mathop\to\limits^f \lim\limits_{\longrightarrow} \Ext^c_R(I^s/J^s,M).
\end{equation}
Then this induces the short exact sequence
\[
0\to H^c_{I}(M)\to H^c_{J}(M)\to N\to 0
\]
where $N:= \im f$. Again apply the functor $\Hom_{R}(\cdot, M)$ to this last sequence we get the following homomorphism
\[
\Ext^c_R(H^c_{J}(M), M)\to \Ext^c_R(H^c_{I}(M), M).
\]
Then from Proposition \ref{31} $(a)$ the statement $(a)$ is clear now. For the proof of $(b)$ note that without loss of generality we may assume that $IR_{\mathfrak{p}}= JR_{\mathfrak{p}}$ for all ${\mathfrak{p}}\in V(J)\cap \Supp_R(M)$ such that
\[
\depth_{R_\mathfrak{p}} (M_{\mathfrak{p}})\leq c.
\]
This is true because the local cohomology is independent up to the radical. We claim that by our assumption on depth we have $\Ext^c_R(I^s/J^s,M)= 0$ for all $s\geq 1$. Since $\Ann_R(I^s/J^s)= J^s:_RI^s$ it is enough to prove that $\grade(J^s:_R I^s,M)\geq c+1$ (see Proposition \ref{31}(a)). Suppose on contrary that $\grade (J^s:_R I^s,M)\leq c.$ It is well-known that
\[
\grade (J^s:_R I^s,M)= \inf \{\depth_{R_{\mathfrak{p}}} (M_{\mathfrak{p}}): {\mathfrak{p}}\in V(J^s:_R I^s)\cap \Supp_R(M)\}
\]
(see \cite[Proposition 1.2.10(a)]{her}). Then there exists a prime ideal
\[
{\mathfrak{p}}\in V(J^s:_R I^s)\cap \Supp_R(M)
\]
such that $\depth_{R_{\mathfrak{p}}} (M_{\mathfrak{p}})\leq c$. On the other side $\Supp_R(R/(J^s:_R I^s))=V(J^s:_R I^s)$ is contained in $V(J)$. It implies that ${\mathfrak{p}}\in V(J)\cap \Supp_R(M)$ and $R_{\mathfrak{p}}/(J_{\mathfrak{p}}^sR_{\mathfrak{p}}:_{R_{\mathfrak{p}}} I_{\mathfrak{p}}^sR_{\mathfrak{p}})\neq 0$ which is a contradiction to our assumption. So we have $\Ext^c_R(I^s/J^s,M)= 0$ for all $s\geq 1$. It implies that $H^c_{I}(M)\cong H^c_{J}(M)$ by virtue of the exact sequence \ref{ee}. Then the statement $(b)$ is easily followed.

To prove the last statement $(c)$ assume that $R$ is complete. Then by $(a)$ and Lemma \ref{3} $(a)$ there is a natural homomorphism
\[
\Hom_{R}(D(H^c_{J}(M)),D(H^c_{J}(M)))\to \Hom_{R}(D(H^c_{I}(M)),D(H^c_{I}(M))).
\]
Suppose that $J\subset I$ satisfy the assumption of $(b)$. For any $R$-module $X$ there is a natural injective homomorphism
\[
\Hom_{R}(X, X)\to \Hom_{R}(D(X),D(X)).
\]
Then consequently it induces the following commutative diagram
\[
\begin{array}{ccc}
   \Hom_{R}(H^c_{J}(M),H^c_{J}(M)) &  \to &  \Hom_{R}(H^c_{I}(M),H^c_{I}(M)) \\
  \downarrow &   & \downarrow \\
   \Hom_{R}(D(H^c_{J}(M)),D(H^c_{J}(M))) & \to & \Hom_{R}(D(H^c_{I}(M)),D(H^c_{I}(M)))
\end{array}
\]
But $\Hom_{R}(H^c_{J}(M),H^c_{J}(M))\to  \Hom_{R}(H^c_{I}(M),H^c_{I}(M))$ and both the vertical homomorphisms are isomorphisms (see $(b)$ and Lemma \ref{3} $(a)$). Then by commutativity of the above diagram it follows that the natural homomorphism
\[
\Hom_{R}(D(H^c_{J}(M)),D(H^c_{J}(M)))\to \Hom_{R}(D(H^c_{I}(M)),D(H^c_{I}(M)))
\]
is an isomorphism. This completes the proof of the Theorem.
\end{proof}
\section{Remarks}
In this section we will give some applications of our main results and we will compare the natural homomorphisms of Theorem \ref{5} and Proposition \ref{21.3}.
\begin{rem}\label{r11}
Let $R$ be a complete Cohen-Macaulay ring of dimension $n$ and $J\subset I$ be two ideals of $R$. Suppose that $\grade(I)=c= \grade(J)$. Then by Proposition \ref{21.3} there are the following commutative diagrams:
\[
\begin{array}{ccc}
  R &  \to & \Hom_{R}(H^c_{J}(K_R),H^c_{J}(K_R)) \\
  || &   & \downarrow \\
  R & \to & \Hom_{R}(H^c_{I}(K_R),H^c_{I}(K_R))
\end{array}
\]
\[
\begin{array}{ccc}
  R &  \to & \Hom_{R}(D(H^c_{J}(K_R)),D(H^c_{J}(K_R))) \\
  || &   & \downarrow \\
   R & \to & \Hom_{R}(D(H^c_{J}(K_R)),D(H^c_{J}(K_R)))
\end{array}
\]
Therefore under the assumption of Proposition \ref{21.3} $(b)$ in case of $M=R$ and from the second last commutative diagram we get that the natural homomorphism
\[
R \to \Hom_{R}(H^c_{J}(K_R),H^c_{J}(K_R))
\]
is an isomorphism if and only if the natural homomorphism
\[
R \to \Hom_{R}(H^c_{I}(K_R),H^c_{I}(K_R))
\]
is an isomorphism. Moreover this is equivalent to the fact that the natural homomorphism
\[
R \to \Hom_{R}(D(H^c_{J}(K_R)),D(H^c_{J}(K_R)))
\]
is an isomorphism if and only if the natural homomorphism
\[
R \to \Hom_{R}(D(H^c_{I}(K_R)),D(H^c_{I}(K_R)))
\]
is an isomorphism (by Lemma \ref{3}).

\end{rem}
\begin{cor}\label{}
Let $J\subset I$ be two ideals of a complete Cohen-Macaulay ring $(R,\mathfrak{m})$ such that $H^i_{J}(R)= 0$ for all $i\neq c= \grade(J,R)= \grade(I,R)$. Suppose that $\Rad IR_{\mathfrak{p}}= \Rad JR_{\mathfrak{p}}$ for all prime ideals ${\mathfrak{p}}\in V(J)$ with $\depth_{R_\mathfrak{p}} (R_\mathfrak{p})\leq c$. Then the natural homomorphism
\[
R \to \Hom_{R}(H^c_{I}(K_R),H^c_{I}(K_R))
\]
is an isomorphism.
\end{cor}
\begin{proof}
It follows from Theorem \ref{5} $(c)$ and the above Remark \ref{r11}.
\end{proof}
\noindent\textbf{Acknowledgement.} The author is grateful to the reviewer for suggestions to
improve the manuscript.


\end{document}